\documentclass[hidelinks,onefignum,onetabnum]{siamart250211}

\usepackage{lipsum} 
\usepackage{amsfonts}
\usepackage{graphicx}
\usepackage{epstopdf}
\ifpdf
  \DeclareGraphicsExtensions{.eps,.pdf,.png,.jpg}
\else
  \DeclareGraphicsExtensions{.eps}
\fi

\newsiamremark{remark}{Remark}
\newsiamremark{hypothesis}{Hypothesis}
\crefname{hypothesis}{Hypothesis}{Hypotheses}
\newsiamthm{claim}{Claim}
\newsiamremark{fact}{Fact}
\crefname{fact}{Fact}{Facts}

\headers{
  Efficient Stochastic BFGS Methods Inspired by Bayesian Principles
}{A. Carlon, L. Espath, and R. Tempone}

\title{
  Efficient Stochastic BFGS methods Inspired by Bayesian Principles
  \thanks{Submitted to the editors on July 10, 2025.}
  }

\author{
  André Carlon
  \thanks{
    Department of Mathematics, RWTH Aachen University, Geb\"{a}ude-1953 1.OG, Pontdriesch 14-16, 161, 52062 Aachen, Germany. 
    (\email{agcarlon@gmail.com}).
  }
  \and Luis Espath
  \thanks{
    School of Mathematical Sciences, University of Nottingham, Nottingham, NG7 2RD, United Kingdom.
  }
  \and Raúl Tempone
  \thanks{
    King Abdullah University of Science \& Technology (KAUST), Computer, Electrical and Mathematical Sciences \& Engineering Division (CEMSE), Thuwal 23955-6900, Saudi Arabia. \\
    Department of Mathematics, RWTH Aachen University, Geb\"{a}ude-1953 1.OG, Pontdriesch 14-16, 161, 52062 Aachen, Germany. \\
    Alexander von Humboldt Professor in Mathematics for Uncertainty Quantification, RWTH Aachen University, Germany.
  }
}

\usepackage{amsopn}

\ifpdf
\hypersetup{
  pdftitle={Efficient Stochastic BFGS methods Inspired by Bayesian Principles},
  pdfauthor={A. Carlon, L. Espath, and R. Tempone}
}
\fi

\usepackage{mathtools}
\usepackage{algpseudocode}
\usepackage{nicefrac}
\usepackage{booktabs}
\usepackage{microtype}
\usepackage{multirow}

\newfont{\tenbfsl}{cmbxti9 scaled 1200}
\newfont{\tenbbb}{msbm10}
\newfont{\svnbbb}{msbm8}

\newcommand{\Exp}[1]{\mathbb{E}\left[ #1 \right]}
\newcommand{\Expc}[2]{\mathbb{E}\left[ \left. #1 \right| #2 \right]}

\newcommand{\bs}[1]{\boldsymbol{#1}}
\newcommand{\cl}[1]{\mathcal{#1}}

\newcommand{\bb}[1]{\mathbb{#1}}

\newcommand{\opt}{\bs{x}^*}
\newcommand{\bigO}[1]{\cl{O}\left( #1 \right)}

\newcommand{\norm}[1]{\left\lVert {#1} \right\rVert}

\newcommand{\inner}[1]{\left\langle {#1} \right\rangle}

\newcommand{\y}{\bs{y}}

\newcommand{\s}{\bs{s}}

\newcommand{\x}{\bs{x}}
\newcommand{\w}{\bs{w}}
\newcommand{\e}{\bs{e}}

\newcommand{\tr}{\textrm{tr}\,}

\newtheorem{thm}{Theorem}
\newtheorem{cor}{Corollary}
\newtheorem{lem}{Lemma}
\newtheorem{prop}{Proposition}
\theoremstyle{remark}

\newcommand{\I}{\bs{I}}
\newcommand{\W}{\bs{W}}

\renewcommand{\sl}{\bs{s}_{\ell}}

\newcommand{\pk}{p_k}

\newcommand{\Hk}{\bs{H}_k}
\newcommand{\Hp}{\bs{H}_{k+1}}
\newcommand{\Bk}{\bs{B}_k}

\newcommand{\z}{\bs{z}}

\begin{document}

\maketitle

\begin{abstract}
  Quasi-Newton methods are ubiquitous in deterministic local search due to their efficiency and low computational cost.
  This class of methods uses the history of gradient evaluations to approximate second-order derivatives.
  However, only noisy gradient observations are accessible in stochastic optimization; thus, deriving quasi-Newton methods in this setting is challenging.
  Although most existing quasi-Newton methods for stochastic optimization rely on deterministic equations that are modified to circumvent noise, we propose a new approach inspired by Bayesian inference to assimilate noisy gradient information and derive the stochastic counterparts to standard quasi-Newton methods. 
  We focus on the derivations of stochastic BFGS and L-BFGS, but our methodology can also be employed to derive stochastic analogs of other quasi-Newton methods.
  The resulting stochastic BFGS (S-BFGS) and stochastic L-BFGS (L-S-BFGS) can effectively learn an inverse Hessian approximation even with small batch sizes.
  For a problem of dimension $d$, the iteration cost of S-BFGS is $\bigO{d^2}$, and the cost of L-S-BFGS is $\bigO{d}$.
  Numerical experiments with a dimensionality of up to $30,720$ demonstrate the efficiency and robustness of the proposed method.
\end{abstract}



\section{Introduction}

Stochastic optimization refers to the optimization of an objective function using stochastic gradient estimates, which are noisy and unbiased approximations of the actual gradient.
In machine learning, the ability to address large-scale problems and the low cost of computing derivatives using back-propagation makes stochastic gradient descent (SGD) and its variants the standard methods for training models.
However, in some cases, objective functions can have much larger second derivatives in some directions than others, causing SGD to zigzag instead of moving directly to the local optimum.
In this setting, knowing second-order information can help SGD converge faster to local minima.
Naively computing Hessian matrices in the stochastic optimization setting by sampling might be unfeasible. 
A large sample size might be needed to obtain a good Hessian estimate, and solving a linear system to get a search direction has complexity $\bigO{d^3}$ for a $d$-dimensional problem.
One approach to ensuring the advantages of second-order methods without the extra cost of computing Hessians is using quasi-Newton methods.

Quasi-Newton methods are among the most successful methods for \emph{deterministic} local search for both convex and nonconvex optimization. 
The idea behind quasi-Newton methods is to use the history of gradient evaluations to construct an approximation of the Hessian matrix of the objective function, combining the low cost of first-order methods with the fast convergence of second-order methods. 
The Broyden-Fletcher-Goldfarb-Shanno (BFGS) method~\cite{broyden1970convergence,fletcher1970new,goldfarb1970family, shanno1970conditioning} is the best-known quasi-Newton method due to its superior performance in comparison to Davidon-Fletcher-Powell (DFP)~\cite{davidon1991variable} and symmetric rank-one update (SR1)~\cite[Section 6.2]{nocedal2006numerical}.
Liu and Nocedal~\cite{liu1989limited} developed a limited-memory version of BFGS (L-BFGS), improving the complexity of each optimization iteration from $\bigO{d^2}$ to $\bigO{d}$.

In the context of stochastic optimization, the absence of noiseless gradient observations renders quasi-Newton methods unsuitable for direct application, as the utilization of noisy estimates in quasi-Newton equations may lead to divergence.
This phenomenon is attributed to the propensity of gradient noise to manifest significantly on approximations of the Hessian and its inverse \cite{xie2020analysis}. 
This problem stems from the fundamental nature of quasi-Newton methods, which aim to satisfy secant equations precisely.
Consequently, in the stochastic setting, these methods inevitably attempt to fit noise. 
In addition, preconditioning the gradient with an inverse Hessian approximation may substantially amplify gradient noise, impeding convergence.

\subsection{Background}
Several approaches have addressed the difficulties of devising quasi-Newton methods for stochastic optimization. For example, Schraudolph, Yu, and G\"unter~\cite{schraudolph2007stochastic} modified the BFGS and L-BFGS equations to make them more robust to noise.
Further, Bordes and Pierre \cite{bordes2009sgd} built on the work of Schraudolph, Yu, and G\"unter, providing an analysis of their method.
In addition, Sohl-Dickstein, Poole, and Ganguli \cite{sohl2014fast} developed a quasi-Newton method for finite-sum minimization mantaining a Hessian approximation for each summand.
Byrd et al.~\cite{byrd2016stochastic} proposed a hybrid algorithm that combines SGD and BFGS, applying low-cost and noisy gradients for most iterations.
After a certain number of iterations, they compute a more precise gradient estimate that to update the inverse Hessian using the standard BFGS equation.
Additionaly, Moritz, Nishihara, and Jordan~\cite{moritz2016linearly} combined the approach by Byrd et al. with the stochastic variance reduced gradient estimator \cite{johnson2013accelerating} to obtain a linear-convergent algorithm in strongly convex smooth problems.

Among the new techniques to use quasi-Newton methods in stochastic optimization, the work by Berahas, Nocedal, and Tak\'a\v{c} \cite{berahas2016multi} is of interest.
They overlapped the samples from subsequent iterations to compute an inverse Hessian approximation using the BFGS equation.
However, this method requires large sample sizes to control the noise in the curvature pairs.
Bollapragada et al.~\cite{bollapragada2018progressive} progressively increased sample sizes to control the gradient noise using the deterministic L-BFGS update.
In addition, Gower, Goldfarb, and Richt\'arik~\cite{gower2016stochastic} proposed a novel method based on sketching. 
A direction is sampled according to an arbitrary rule, and the inverse Hessian update is forced to fit the actual Hessian action in this direction while moving as little as possible in a weighted Frobenius norm sense.
Further, Wang et al.~\cite{wang2017stochastic} modified the BFGS update to mitigate the effect of noise on the Hessian approximation and proposed a limited-memory version of this algorithm named Stochastic-damped L-BFGS (SdLBFGS).
More recently, Berahas et al.~\cite{berahas2022quasi},instead of using the optimization history to approximate a Hessian, sampled several points near the current iterate and computed the gradients differences between them, using these differences to compute an inverse Hessian approximation using standard L-BFGS equations.
In general, quasi-Newton approaches for stochastic optimization rely on adapting deterministic equations to the stochastic setting by increasing sample sizes, using specific sampling strategies, or slightly modifying to the equations to make them more robust.

To the authors' knowledge, only three approaches exist to devise quasi-Newton methods for stochastic optimization from first principles, two of which use Bayesian inference to model learning about the Hessian from curvature pair data~\cite{hennig2013fast, carlon2024approximating}.
The methods differ in how they model the prior and likelihood distributions, resulting in different quasi-Newton methods.
In the third work, by Li~\cite{li2017preconditioned}, the author derives different strategies for constructing preconditioning matrices by considering convergence improvement and noise amplification.

\subsection{Contributions}
This study follows the work by Carlon, Espath, and Tempone~\cite{carlon2024approximating} in the Bayesian modeling of prior and likelihood distributions.
However, instead of approximating the Hessian, our method seeks preconditioning matrices that satisfy secant equations while avoiding the noise amplification problem.
Moreover, in~\cite{carlon2024approximating}, the authors used log barriers to enforce eigenvalue constraints, whereas we do so by accepting/rejecting curvature pairs.
Thus, we are able to derive closed-form solutions for the derivation of quasi-Newton methods, avoiding the expensive optimization subproblem that the authors in~\cite{carlon2024approximating} solve to maximize the posterior distribution.
Specifically, We focus on deriving stochastic counterparts to BFGS (S-BFGS) and L-BFGS (L-S-BFGS).
However, the same methodology can be employed to derive stochastic analogues to DFP, good Broyden, and bad Broyden~\cite{nocedal2006numerical} by choosing appropriate weight matrices and prior and likelihood distributions.
These equations are a special case of quasi-Newton methods using weighted secant equations~\cite{gratton2015quasi}, in which each curvature pair is weighted by its precision.
The implementation of the derived S-BFGS and L-S-BFGS methods is straightforward and requires tuning only a few parameters to function effectively.

To validate our method, we compare L-S-BFGS against SdLBFGS~\cite{wang2017stochastic} and oLBFGS~\cite{mokhtari2015global} on training a logistic regression model on four different datasets.
In all cases, L-S-BFGS outperforms the baseline methods, demonstrating the efficiency and robustness of our approach.
The cost of L-S-BFGS increases linearly with the problem dimensionality; hence, L-S-BFGS is suitable for large-scale stochastic optimization problems.

\subsection{Problem formulation}
In what follows, we employ bold lowercase symbols for vectors, bold uppercase symbols for matrices, and regular lowercase symbols for scalars.

Let $f : \bb{R}^d \times \Omega \rightarrow \bb{R}$ be the objective function of a stochastic optimization problem,
\begin{equation}
    \label{eq:problem}
    \tag{P}
    \text{find }
    \bs{x}^* = \underset{\bs{x} \in \bb{R}^d}{\arg \min} \; \Exp{f(\bs{x}, \bs{\xi})},
\end{equation}
where $\bs{x}$ denotes the optimization variables and $\bs{\xi}$ representes a random vector.
One approach to solving \eqref{eq:problem} is using SGD,
\begin{equation}
    \bs{x}_{k+1} = \bs{x}_k - \eta_k \underbrace{\frac{1}{N} \sum_{n=1}^N \nabla_{\bs{x}}  f(\bs{x}_k, \bs{\xi}_{k, n})}_{\bs{\upsilon}_k},
\end{equation}
where each $\bs{\xi}_{k, n}$ is sampled independently with an identical distribution for each iteration, $\eta_k>0$ is the step size, and $N$ is the batch size.
The convergence of SGD depends on the problem characteristics and the choice of step sizes $\eta_k$.
The aim is to find a matrix $\Hk \in \bb{R}^{d \times d}$ that can be used to precondition the gradient estimate in SGD,
\begin{equation}
    \label{eq:psgd}
    \bs{x}_{k+1} = \bs{x}_k - \eta_k \Hk \bs{\upsilon}_k,
\end{equation}
hopefully improving its convergence.
In the spirit of quasi-Newton methods, information from curvature pairs generated during optimization is employed to build $\Hk$, namely $\s_k = \bs{x}_k - \bs{x}_{k-1}$ and $\y_k = \frac{1}{N} \sum_{n=1}^N( \nabla_{\bs{x}}  f(\bs{x}_k, \bs{\xi}_{k, n}) -  \nabla_{\bs{x}}  f(\bs{x}_{k-1}, \bs{\xi}_{k, n}))$.
A Bayesian formulation is applied to incorporate curvature information into a preconditioning matrix $\Hk$, where each pair has a precision scalar $\pk>0$ that encodes the dispersion of $\y_k$.
\section{Bayesian formulation of stochastic BFGS}

\subsection{Bayesian formulation of quasi-Newton methods}

We use Bayesian inference to build a posterior distribution of the Hessian inverse $\bs{H}$ as
\begin{equation} \label{eq:bayes}
  \pi(\bs{H} | \s_k, \y_k) \propto p(\y_k | \bs{H}, \s_k) \pi(\bs{H}),
\end{equation}
where $p(\y_k | \bs{H}, \s_k)$ is the likelihood of observing the curvature pair $(\s_k, \y_k)$ given a matrix $\bs{H}$ and $\pi(\bs{H})$ is a prior distribution of $\bs{H}$.
One approach to find an approximation of $\bs{H}$ is to use the \emph{maximum a posteriori} (MAP), the matrix that maximizes the posterior in \eqref{eq:bayes}.
Let us model the negative log prior distribution of $\bs{H}$ using a weighted Frobenius distance to the current Hessian inverse approximation $\Hk$,
\begin{align}
    - \log(\pi(\bs{H}))
    &= \frac{1}{2} \norm{\bs{H} - \bs{H}_k}_{F, \W_{pr}}^2 + C_{pr} \\
    &= \frac{1}{2} \norm{\sqrt{\W_{pr}}(\bs{H} - \bs{H}_k)\sqrt{\W_{pr}}}_{F}^2 + C_{pr},
\end{align}
where $\W_{pr} \in \bb{R}^{d \times d}$ is an appropriate positive-definite matrix such that the log prior is dimensionless and $C_{pr}$ is a constant that guarantees that the prior is a proper distribution.
This prior distribution selection is not ad hoc; $\pi(\bs{H})$ is a Gaussian-type distribution.
Moreover, this negative log prior is exactly the objective function used to derive the BFGS method~\cite{goldfarb1970family}.
Regarding the likelihood, we model the negative log likelihood of observing the curvature pair $(\s_k, \y_k)$ as
\begin{align}
    - \log(p(\y_k | \bs{H}, \s_k))
    &= \frac{1}{\rho} \norm{\s_k - \bs{H} \y_k}_{\W_l}^2 + C_l \\
    &=  \frac{1}{\rho} \norm{\sqrt{\W_l}(\s_k - \bs{H} \y_k)}^2 + C_l,
\end{align}
where $\W_l \in \bb{R}^{d \times d}$ is a positive-definite matrix, $\rho>0$ is a parameter to be tuned, and $C_l$ is a constant.
As with the prior distribution, the choices of $\W_l$, $\rho$, and $C_l$ must ensure that the likelihood is a proper probability distribution.

The problem of determining the current Hessian inverse approximation $\bs{H}_k$ as the \emph{maximum a posteriori}(MAP) can be written as
\begin{equation}\label{eq:opt_subproblem}
    \begin{aligned}
    \bs{H}_{k+1} = &\underset{\bs{H} \in \bb{R}^{d \times d}}{\arg \min}
    \left(
        \frac{1}{\rho} \norm{\s_k - \bs{H} \y_k}_{\W_l}^2
        +
        \frac{1}{2} \norm{\bs{H} - \bs{H}_k}_{F, \W_{pr}}^2
    \right) \\
    & \text{s.t. } \bs{H} = \bs{H}^\top.
    \end{aligned}
\end{equation}

Analyzing~\eqref{eq:opt_subproblem} reveals that the negative log posterior is a strongly convex quadratic function on $\bs{H}$.
Thus, this strongly convex quadratic programming problem with a linear constraint has a unique solution.
Figure~\ref{fig:bayesian_representation} graphically represents the Bayesian formulation for quasi-Newton methods.

\begin{figure}[t]
    \centering
    \includegraphics[width=\linewidth]{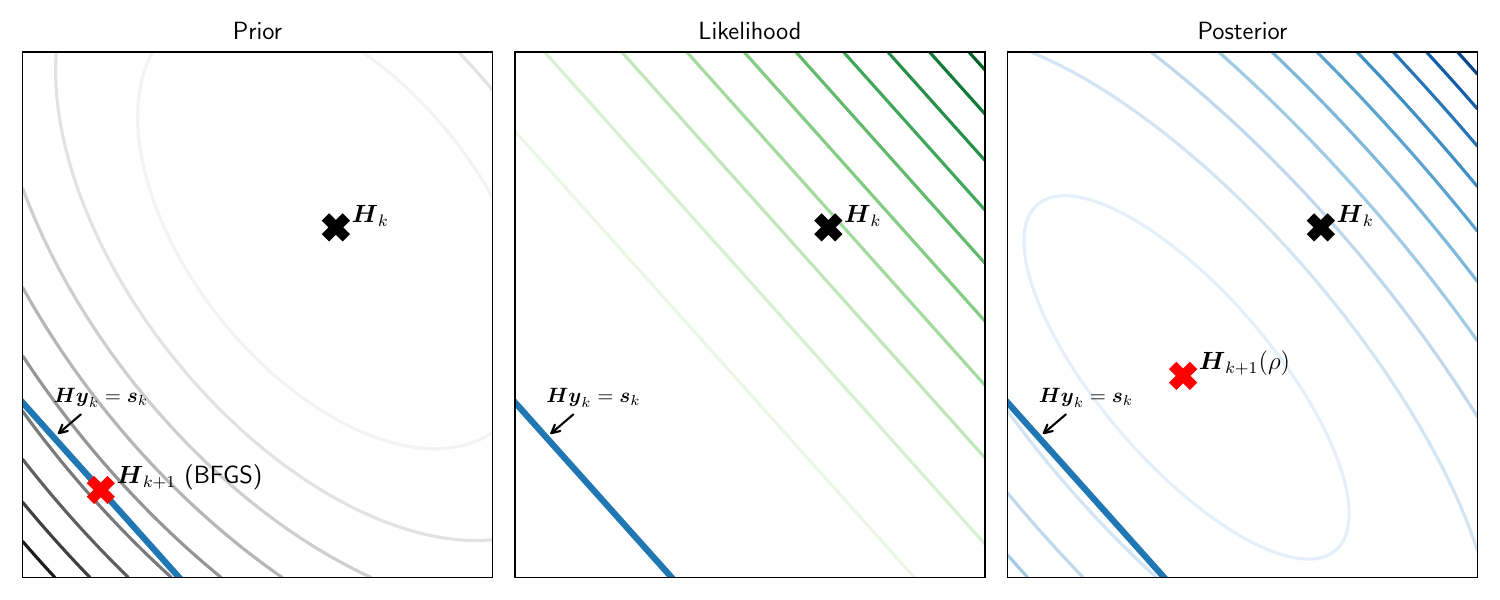}
    \caption{Graphical representation of the Bayesian formulation for quasi-Newton methods. 
    Left: Contour of the negative log prior distribution of $\bs{H}$, with a blue line representing the affine subspace of matrices satisfying the secant equation and a red $\times$ marking the BFGS update. 
    Center: Contour of the negative log-likelihood distribution. A larger confidence on the observed $\y_k$ indicates that it is more likely that the true Hessian is near the affine subspace $\bs{H} \y_k = \s_k$. 
    Right: Contour of the negative log posterior for a given $\rho$ and the $\Hp$ that minimizes it. 
    A larger $\rho$ results in $\Hp$ being closer to $\Hk$, whereas as $\rho \downarrow 0$, the new inverse Hessian approximation $\Hp$ converges to the one of BFGS.
    }
    \label{fig:bayesian_representation}
\end{figure}

\subsection{Derivation of the S-BFGS method}

As in deterministic quasi-Newton methods, different choices of the weight matrices result in different quasi-Newton methods.
Let $\W$ be any positive-definite matrix satisfying $\W \s_k = \y_k$.
Then, letting $\W_{pr} = \W$ and $\W_l = p_k \W$, where $p_k>0$ is a scalar representing the precision of the pair $(\s_k, \y_k)$, the solution of problem~\eqref{eq:opt_subproblem} furnishes a stochastic BFGS analogue, which we name S-BFGS.
If one exchanges $\bs{H}$ for $\bs{B}$, $\s_k$ for $\y_k$, and sets $\W$ such that $\W \y_k = \s_k$, the solution of problem~\ref{eq:opt_subproblem} results in a stochastic DFP method.
Here, we focus on the S-BFGS method.
The variable $\pk$ is set to be the inverse of the trace of the covariance matrix of $\y_k$. The trace of the covariance of $\y_k$ is the componentwise sum of the variances. 
This choice of $\W_{pr}$ ensures the prior distribution is valid because the log prior is dimensionless. 
Other choices for $\W_{pr}$, such as $\W_{pr}=\I$, may not be consistent. 
Similarly, the parameter $\rho$ is chosen to ensure that the likelihood remains a valid probability distribution, aligning with the probabilistic interpretation of the Bayesian framework.
In Proposition~\ref{prp:sbfgs_update} we derive the S-BFGS update.

\begin{prop} \label{prp:sbfgs_update}
    Let $\W \in \bb{R}^{d \times d}$ be any positive-definite matrix such that $\W \s_k = \y_k$.
    Then, the update of the stochastic BFGS (S-BFGS), defined as the solution of the minimization problem~\eqref{eq:opt_subproblem} with $\W_{pr}=\W$ and $\W_l=p_k \W$, can be written as
    \begin{align}
        \label{eq:sbfgs}
        \Hp = \Hk
        + \frac{
            1 + \frac{\y_k^\top \Hk \y_k}{\s_k^\top \y_k + \frac{\rho}{p_k}}
        }{
            \s_k^\top \y_k + \frac{\rho}{2 p_k}
        } 
        \s_k \s_k^\top
        - \frac{1}{\s_k^\top \y_k + \frac{\rho}{p_k }} 
        (\Hk \y_k \s_k^\top + \s_k \y_k^\top \Hk).
    \end{align}
\end{prop}

\begin{proof}
    From the first-order optimality condition of the problem in \eqref{eq:opt_subproblem}, $\Hp$ must satisfy
    \begin{multline}
         \pk \left(\W (\Hp \y_k - \s_k) \y_k^\top + \y_k (\Hp \y_k - \s_k)^\top \W \right) \\
        + \rho \left(\W (\Hp - \Hk) \W\right) = \bs{0}_d,
    \end{multline}
    where $\bs{0}_d \in \bb{R}^{d \times d}$ is a matrix filled with zeros.
    Multiplying both sides from the left and right by $\W^{-1}$ and noting that $\W^{-1} \y_k = \s_k$ yields the following continuous Lyapunov equation,
    \begin{align} \label{eq:lyap_eq}
        \Hp \left(
            \y_k \s_k^\top + \frac{\rho}{2 p_k } \I_d
        \right)
        +
        \left(
            \s_k \y_k^\top + \frac{\rho}{2 p_k } \I_d
        \right)
        \Hp
        = 2 \s_k \s_k^\top + \frac{\rho}{p_k } \Hk.
    \end{align}
    The matrices $\y_k \s_k^\top + \frac{\rho}{2 p_k } \I_d$ and $2 \s_k \s_k^\top + \frac{\rho}{p_k } \Hk$ are positive-definite; hence, this system has a unique solution, and it is positive-definite.
    Let us assume an update of the form
    \begin{align}
        \label{eq:H_def}
        \Hp = \Hk + a \s_k \s_k^\top + b (\Hk \y_k \s_k^\top + \s_k \y_k^\top \Hk).
    \end{align}
    Plugging $\Hp$ into \eqref{eq:lyap_eq} and solving for the constants $a$ and $b$ finishes the proof.
\end{proof}

Notice that letting the precision of the curvature pair $p_k \rightarrow \infty$ recovers the deterministic BFGS update.
Thus, the deterministic BFGS method is a particular case of S-BFGS arising when the noise in the gradient curvature approaches zero.
Similarly, we recover the BFGS if we let $\rho \downarrow 0$.
This choice of $\W_{pr}$ and $\W_l$ results in a special case of BFGS with weighted secant equations~\cite{gratton2015quasi}, where the weights are given by the precision $p_k$.

Algorithm~\ref{alg:sbfgs} presents pseudocode for the S-BFGS.

\begin{algorithm}
    \caption{Pseudocode for the stochastic BFGS}
    \label{alg:sbfgs}
    \begin{algorithmic}
        \Procedure{S-BFGS}{$\Hk$, $\s_k$, $\y_k$, $p_k$, $\rho$}
        \State $a_k \gets \frac{
            1 + \frac{\y_k^\top \Hk \y_k}{\s_k^\top \y_k + \frac{\rho}{p_k }}
            }{
                \s_k^\top \y_k + \frac{\rho}{2 p_k }
                }$
        \State $b_k \gets -\frac{1}{\y_k^\top \s_k + \frac{\rho}{p_k }}$
        \State $\Hp \gets \Hk + a_k \s_k \s_k^\top + b_k (\Hk \y_k \s_k^\top + \s_k \y_k^\top \Hk)$
        \State \Return $\Hp$
        \EndProcedure
    \end{algorithmic}
\end{algorithm}

The S-BFGS method has memory allocation and processing cost that are $\bigO{d^2}$.
In order to tackle large-scale problems, we derive a limited-memory variation of the S-BFGS in the next section.

\subsection{Limited-memory Stochastic BFGS}
The key idea behind L-BFGS is to compute the action of the Hessian inverse approximation on any given vector directly, avoiding assembling the full matrix~\cite{liu1989limited}.
Here we follow a similar approach.
Let $\z \in \bb{R}^d$ be any arbitrary vector.
We are interested in devising a method to directly compute $\Hp \z$ without assembling $\Hp$.
In the same spirit as L-BFGS, $r$ curvature pairs are kept in memory and used to approximate $\Hp \z$.
Algorithm \ref{alg:lsbfgs} presents the pseudocode for the limited-memory S-BFGS.
Because L-S-BFGS has two nested loops over the curvature pairs, its cost is $\bigO{d r^2}$.

\begin{algorithm}[h]
    \caption{Limited-memory version of stochastic BFGS (L-S-BFGS)}
    \label{alg:lsbfgs}
    \begin{algorithmic}
        \Procedure{L-S-BFGS}{$\left\{(\s_\ell, \y_\ell, p_\ell)\right\}_{\ell=k-r+1}^k$, $\z$, $\bs{H}_0$, $\rho$}
        \State $\bs{w} \gets \bs{H}_0 \z$
        \Comment{$\bs{w}$ is a placeholder for $\Hp \z$}
        \For{$i=k-r+1$ to $k$}
            \State $\bs{v}_i \gets \bs{H}_0 \y_i$
            \Comment{$\bs{v}_i$ is a placeholder for $\bs{H}_{i} \y_i$}
            \For{$j=k-r$ to $i-1$}
                \State $\bs{v}_i \gets \bs{v}_i + 
                a_j \s_j \s_j^\top \y_i
                + b_j (\bs{v}_j \s_j^\top \y_i + \s_j \bs{v}_j^\top \y_i)$
            \EndFor
            \State $a_i \gets \frac{
                1 + \frac{\y_i^\top \bs{v}_i}{\s_i^\top \y_i + \frac{\rho}{p_i }}
                }{
                    \s_i^\top \y_i + \frac{\rho}{2 p_i }
                    }$
            \State $b_i \gets -\frac{1}{\y_i^\top \s_i + \frac{\rho}{p_i }}$
            \State $\bs{w} \gets \bs{w} + a_i \s_i \s_i^\top \bs{z}
            + b_i (\bs{v}_i \s_i^\top \bs{z} + \s_i \bs{v}_i^\top \bs{z})$
        \EndFor
    \State \Return $\bs{w}$
    \EndProcedure
    \end{algorithmic}
\end{algorithm}

In addition to the reduced memory and processing costs, L-S-BFGS has an advantage over S-BFGS in terms of forgetting old data; S-BFGS must sometimes be restarted, whereas L-S-BFGS does not.

\subsection{Curvature condition for stochastic optimization}

Quasi-Newton methods for stochastic optimization encounter two significant challenges.
First, these methods must obtain accurate information on the Hessian and its inverse from noisy gradient evaluations.
Second, they must avoid amplifying the gradient noise, which can cause the optimization process to diverge, particularly when the preconditioning matrix has large eigenvalues. 
To prevent this, we propose a rule for selecting curvature pairs that avoids large eigenvalues for the inverse Hessian, thereby avoiding noise amplification.
The proposed curvature condition is stronger than the curvature condition of deterministic quasi-Newton methods.
While deterministic quasi-Newton methods often accept a curvature pair $(\s_k, \y_k)$ if it satisfies $\y_k^\top \s_k > 0$, we only accept the curvature pair if
\begin{equation}
    \label{eq:curv_condition}
    \y_k^\top \s_k \ge m \norm{\s_k}^2,
\end{equation}
for a parameter $m>0$.
The intuition behind \eqref{eq:curv_condition} is that a function $\tilde{F}: \bb{R}^d \rightarrow \bb{R}$ being $m$-convex is equivalent to
\begin{equation}
    \inner{
        \nabla \tilde{F}(\bs{x}) - \nabla \tilde{F}(\bs{y}),
        \bs{x} - \bs{y}}
    \ge
    m \norm{\bs{x} - \bs{y}}^2 ,
\end{equation}
for all $\bs{x}, \bs{y} \in \bb{R}^d$.
Moreover, if $\tilde{F}$ is also twice differentiable, $(\nabla^2 \tilde{F}(\bs{x}))^{-1} \preceq \nicefrac{1}{m}$ for all $\bs{x}$~\cite{nesterov2013introductory}.
Although a large value of $m$ might improve the stability of S-BFGS and S-L-BFGS by avoiding larger eigenvalues in $\Hk$, it also results in data loss as many curvature pairs are discarded.
The choice of $m$ depends on the noise level on the gradient evaluations and the precision sought after in the optimum found.
Moreover, a large step size might require a larger, more conservative $m$.
In practical numerical examples, we have found advantageous to maintain a large step size and control the stability of the methods by increasing $m$ and $\rho$.

In addition to controlling the largest eigenvalues of $\bs{H}$, its smallest eigenvalues can be controlled using the following extra curvature condition,
\begin{equation}
    \label{eq:extra_curv_condition}
    \y_k^\top \s_k \le M \norm{\s_k}^2,
\end{equation}
for $M>m>0$
Although not as important as the condition in~\eqref{eq:curv_condition}, iteratively controlling the smallest eigenvalues of the preconditioning matrix might avoid optimization steps that are too small in some directions.
This extra curvature condition might be helpful when the Lipschitz smoothness of the objective function $L$ is known; therefore, $m$ can be set to $L$.

Next, we establish several properties of our S-BFGS update under the proposed curvature conditions.
Many arguments in Lemmas~\ref{lem:pd} and \ref{lem:bounds} and in Corollary~\ref{cor:eigen_bounds} are adapted from \cite{byrd1989tool} to fit our algorithm.

\begin{lem}[Properties of the S-BFGS update]
    \label{lem:pd}
    If $\Hk$ is positive definite, the update of S-BFGS~\eqref{eq:sbfgs} with curvature condition $m \norm{\s_k}^2 \le \y_k^\top \s_k$ generates a matrix $\Hp$ that is also positive definite.
\end{lem}

\begin{proof}
    Let the update of S-BFGS be written as
    \begin{align}
        \Hp = (I + b \s_k \y_k^\top) \Hk (I + b \y_k \s_k^\top) + (a - b^2 \y_k^\top \Hk \y_k) \s_k \s_k^\top.
    \end{align}
    Next, we prove that $a - b^2 \y_k^\top \Hk \y_k >0$,
    \begin{align}
        a - b^2 \y_k^\top \Hk \y_k
        &= \frac{
            1 + \frac{\y_k^\top \Hk \y_k}{\s_k^\top \y_k + \frac{\rho}{p_k }}
            }{
                (\s_k^\top \y_k + \frac{\rho}{2 p_k})
            }
            - \frac{\y_k^\top \Hk \y_k}{(\s_k^\top \y_k + \frac{\rho}{p_k})^2} \\
        &>
        \frac{
            1 + \frac{\y_k^\top \Hk \y_k}{\s_k^\top \y_k + \frac{\rho}{p_k }}
            }{
                (\s_k^\top \y_k + \frac{\rho}{p_k})
            }
            - \frac{\y_k^\top \Hk \y_k}{(\s_k^\top \y_k + \frac{\rho}{p_k})^2} \\
        \label{eq:upper_bound}
        &= \frac{1}{\s_k^\top \y_k + \frac{\rho}{p_k}} > 0,
    \end{align}
    where the curvature condition is used in the last step. For any vector $\z$,
    \begin{align}
        \label{eq:psd}
        \z^\top \Hp \z = \w^\top \Hk \w + (a - b^2 \y_k^\top \Hk \y_k) (\z^\top \s_k)^2 > 0,
    \end{align}
    where $\w = \z + b \y_k \s_k^\top \z$, proving the positive-definiteness of $\Hp$.
    It is not possible for both terms on the right-hand side to vanish simultaneously because, if $\z^\top \s_k = 0$, then $\w=\z$.
\end{proof}

\begin{lem}
    \label{lem:bounds}
    For the update of S-BFGS~\eqref{eq:sbfgs} with curvature conditions $m \norm{\s_k}^2 \le \y_k^\top \s_k \le M \norm{\s_k}^2$, $\Hp$ has its trace bounded above as
    \begin{align}
        \tr(\Hp)
        \le \tr(\Hk) + 
        \left(
        \frac{
            1 + \frac{\y_k^\top \Hk \y_k}{m \norm{\s_k}^2 + \frac{\rho}{p_k}}
        }{
            m \norm{\s_k}^2 + \frac{\rho}{2 p_k}
        } 
        \right)
        \norm{\s_k}^2
        - 
        \frac{2 \s_k^\top \Hk \y_k}{M \norm{\s_k}^2 + \frac{\rho}{p_k }},
    \end{align}
    and its determinant is bounded below as
    \begin{align}
        \det(\Hp) 
        &> \det(\Hk) \left(
            \frac{\rho^2}{p_k^2\left(\s_k^\top \y_k + \frac{\rho}{p_k}\right)^2}
            +\frac{\s_k^\top \Bk \s_k}{\s_k^\top \y_k + \frac{\rho}{p_k}}\right).
    \end{align}
\end{lem}
\begin{proof}
    The trace is a linear operator, $\tr(\bs{A} + \bs{B}) = \tr(\bs{A}) + \tr(\bs{B})$; thus, using the curvature conditions yields
    \begin{align}
        \tr(\Hp)
        & = \tr(\Hk) + 
        \left(
        \frac{
            1 + \frac{\y_k^\top \Hk \y_k}{\s_k^\top \y_k + \frac{\rho}{p_k}}
        }{
            \s_k^\top \y_k + \frac{\rho}{2 p_k}
        } 
        \right)
        \norm{\s_k}^2
        - 
        \frac{2 \s_k^\top \Hk \y_k}{\s_k^\top \y_k + \frac{\rho}{p_k }} \\
        & \le \tr(\Hk) + 
        \left(
        \frac{
            1 + \frac{\y_k^\top \Hk \y_k}{m \norm{\s_k}^2 + \frac{\rho}{p_k}}
        }{
            m \norm{\s_k}^2 + \frac{\rho}{2 p_k}
        } 
        \right)
        \norm{\s_k}^2
        - 
        \frac{2 \s_k^\top \Hk \y_k}{M \norm{\s_k}^2 + \frac{\rho}{p_k }}.
    \end{align}
    First, we consider the following identity to bound the determinant,
    \begin{align}
    \det(\I + \bs{v}_1 \bs{v}_2^\top + \bs{v}_3 \bs{v}_4^\top) = (1 + \bs{v}_2^\top \bs{v}_1)(1 + \bs{v}_4^\top \bs{v}_3) - \bs{v}_4^\top \bs{v}_1 \bs{v}_2^\top \bs{v}_3,
    \end{align}
    which was also used in~\cite{pearson1969variable}.
    Then, from \eqref{eq:H_def},
    \begin{align}
        \det(\Hp) 
        &= \det(\Hk)\det(\I + a \Bk \s_k \s_k^\top + b(\y_k \s_k^\top + \Bk \s_k \y_k^\top \Hk)) \\
        &= \det(\Hk)\det(\I + \Bk \s_k (a \s_k^\top + b \y_k^\top \Hk) + b\y_k \s_k^\top) \\
        &= 
        \begin{multlined}
            \det(\Hk)
            [
            (1 + (a \s_k^\top + b \y_k^\top \Hk) \Bk \s_k)(1 + b \s_k^\top \y_k)
            \\
            -
            b \s_k^\top \Bk \s_k (a \s_k^\top + b \y_k^\top \Hk) \y_k
            ] 
        \end{multlined}
        \\
        &= \det(\Hk)
        \left[
            (1 + b \y_k^\top \s_k)^2 + \s_k^\top \Bk \s_k (a - b^2 \y_k^\top \Hk \y_k)
        \right].
    \end{align}
    Using \eqref{eq:upper_bound} and substituting $a$ and $b$ concludes the proof.
\end{proof}

\begin{cor}[Bounds on the eigenvalues of preconditioning matrices]
    \label{cor:eigen_bounds}
If the assumptions of Lemma~\ref{lem:bounds} are satisfied, for a finite number $k+1$ of iterations of S-BFGS, constants $0 < \tilde{\mu}, \tilde{L} < +\infty$ exist such that $(\tilde{L})^{-1} \preceq \bs{H}_i \preceq (\tilde{\mu})^{-1}$ for $i=0,1,\dots,k+1$.
\end{cor}

\begin{proof}
    Let $\lambda_j^{(k+1)}$ be the eigenvalues of $\Hp$ sorted in ascending order. 
    Then, Lemma~\ref{lem:bounds} yields
    \begin{align}
        \lambda_d^{(k+1)}
        &\le \tr(\Hp) \\
        &\le         
        \tr(\Hk) +
        \left(
            \frac{
                1 + \frac{\y_k^\top \Hk \y_k}{m \norm{\s_k}^2 + \frac{\rho}{p_k}}
            }{
                m \norm{\s_k}^2 + \frac{\rho}{2 p_k}
            } 
            \right)
            \norm{\s_k}^2
            - 
            \frac{2 \s_k^\top \Hk \y_k}{M \norm{\s_k}^2 + \frac{\rho}{p_k }} \\
        &\le
        \tr(\bs{H}_0) +
        \sum_{i=0}^k
        \left[
        \left(
            \frac{
                1 + \frac{\y_i^\top \bs{H}_i \y_i}{m \norm{\s_i}^2 + \frac{\rho}{p_i}}
            }{
                m \norm{\s_i}^2 + \frac{\rho}{2 p_i}
            } 
            \right)
            \norm{\s_i}^2
            - 
            \frac{2 \s_i^\top \bs{H}_i \y_i}{M \norm{\s_i}^2 + \frac{\rho}{p_i }} 
            \right].
    \end{align}
    Similarly, let $\z_{k+1}$ be the correspondent eigenvector to the smallest eigenvalue of $\Hp$.
    Using \eqref{eq:psd},
    \begin{align}
        \lambda_0^{(k+1)}
        &=\z_{k+1}^\top \Hp \z_{k+1} \\
        &\ge 
        \lambda_0^{(k)} \norm{\z_{k+1} + b_k \y_k \s_k^\top \z_{k+1}}^2
        + \left(\frac{1}{\s_k^\top \y_k + \frac{\rho}{p_k}}\right) (\z_{k+1}^\top \s_k)^2,
    \end{align}
    thus, there exists a constant $\tilde{L}>0$ such that $\lambda_0^{(i)} \ge (\tilde{L})^{-1}$ for $i=0,1,\dots,k+1$.
\end{proof}

\section{Convergence and cost analysis}

In this section, we present a convergence analysis of SGD preconditioned by a positive-definite matrix for strongly convex smooth problems.
First, we prove in Lemma~\ref{lem:precond} some properties of preconditioned SGD.
Then, Theorem~\ref{thm:conv} uses Lemma~\ref{lem:precond} to bound the optimality gap above and establish conditions for the preconditioned SGD to reach a prescribed error tolerance $\epsilon$.
As observed in Lemma~\ref{lem:pd} and Corollary~\ref{cor:eigen_bounds}, the curvature conditions imposed on S-BFGS and L-S-BFGS guarantee that $\Hp$ remains positive definite with bounded positive eigenvalues. 
Thus, S-BFGS and L-S-BFGS satisfy the assumptions of Lemma~\ref{lem:precond} and Theorem~\ref{thm:conv}.

\begin{lem}[preconditioned SGD]
    \label{lem:precond}
    Assume a preconditioning SGD update \eqref{eq:psgd} with a positive definite preconditioning matrix $\Hk$ with bounded eigenvalues $0 < (\tilde{L})^{-1} \le \norm{\bs{H}_k} \le (\tilde{\mu})^{-1} < +\infty$ and let $\w_k^{(i)} = (\sqrt{\bs{H}_i})^{-1} \x_k$.
    Also, let $g_i(\w_k^{(i)}) \coloneqq F(\sqrt{\bs{H}_i} \w_k^{(i)}) = F(\x_k)$.
    Then, by setting $\e_k$ as the gradient error at iteration $k$, the preconditioned SGD update,
    \begin{align}
        \x_{k+1} &= \x_k - \eta_k \Hk \bs{\upsilon}_k \\
                 &= \x_k - \eta_k \Hk (\nabla F(\x_k) + \e_k)
    \end{align}
    is a linear transformation of the vanilla SGD with the gradient estimate error premultiplied by $\sqrt{\Hk^\top}$,
    \begin{align}
        \w_{k+1}^{(k)} = \w_k^{(k)} - \eta_k \left(\nabla g_k(\w_k^{(k)}) + \sqrt{\Hk^\top} \e_k \right).
    \end{align}
    Further, if $F$ is twice differentiable, strongly convex and $\Hk$ is positive definite, $g_k$ is also twice differentiable, strongly convex, and has Hessian
    \begin{align}
        \label{eq:hess_g}
        \nabla^2 g_k(\w_k^{(k)}) = \sqrt{\Hk^\top} \nabla^2 F(\x_k) \sqrt{\Hk},
    \end{align}
    which satisfies
    \begin{align}
        \label{eq:g_bounds}
        \frac{\mu}{\tilde{L}}
        \preceq
        \nabla^2 g_k(\w_k^{(k)})
        \preceq
        \frac{L}{\tilde{\mu}}.
    \end{align}
\end{lem}

\begin{proof}
    As $\nabla g_k(\w_k^{(k)}) =  \sqrt{\Hk^\top} \nabla F(\sqrt{\Hk} \w_k^{(k)})$,
    \begin{align}
        \w_{k+1}^{(k)} 
        &= \w_k^{(k)} - \eta_k (\nabla g_k(\w_k^{(k)}) + \sqrt{\Hk^\top} \e_k) \\
        &= \w_k^{(k)} - \eta_k  \sqrt{\Hk^\top} \nabla F(\sqrt{\Hk} \w_k^{(k)}) - \eta_k \sqrt{\Hk^\top} \e_k
    \end{align}
    Premultiplying both sides by $\sqrt{\Hk}$ yields
    \begin{align}
        \sqrt{\Hk} \w_{k+1}^{(k)}
        &= \sqrt{\Hk} \w_k^{(k)} - \eta_k  \sqrt{\Hk} \sqrt{\Hk^\top} \nabla F(\sqrt{\Hk} \w_k^{(k)})
        - \eta_k \Hk \e_k
        \\
        \x_{k+1} &= \x_k - \eta_k \Hk (\nabla F(\x_k) + \e_k).
    \end{align}
    Next, taking the second derivatives of $g$,
    \begin{align}
        \nabla^2 g_k(\w_k^{(k)}) = \sqrt{\Hk^\top} \nabla^2 F(\sqrt{\Hk} \w_k^{(k)}) \sqrt{\Hk}.
    \end{align}
    We obtain \eqref{eq:hess_g} because $\sqrt{\Hk} \w_k^{(k)}=\x_k$.
    Using the spectral norm inequalities above and below yields \eqref{eq:g_bounds}, proving that $\nabla^2 g_k$ is positive definite, thus $g$ is strongly convex.
\end{proof}

\begin{thm}[Convergence of the preconditioned SGD]
    \label{thm:conv}
    Let $F: \bb{R}^d \rightarrow \bb{R}$ be twice differentiable, $L$-smooth, and $\mu$-strongly convex, and let $\epsilon>0$ be an arbitrary tolerance on the expected optimality gap.
    Assume access to an unbiased stochastic oracle $\bs{\upsilon}_k$ such that $\Expc{\bs{\upsilon}_k}{\x_k} = \nabla F(\x_k)$, and consider the preconditioned SGD update
    \begin{align}
        \x_{k+1} = \x_k - \eta_k \Hk \bs{\upsilon}_k,
    \end{align}
    where $\Hk$ is a positive definite matrix satisfying $\norm{\Hk} \le (\tilde{\mu})^{-1}$ and
    \begin{align}
        \label{eq:bounds_g}
        \hat{\mu} \preceq \sqrt{\Hk^\top} \nabla^2 F(\x_k) \sqrt{\Hk} \preceq \hat{L}
        \quad \textrm{for all }k.
    \end{align}
    If the step size is set at a fixed value,
    \begin{align}
        \eta_k = \eta \le \min\left\{
            \frac{1}{\hat{L}},
            \frac{4}{\hat{L}}
            \left(
            \frac{1}{
                \frac{1}{\hat{\mu} \tilde{\mu}} 
                \left(\frac{\sigma^2}{(1-\alpha) \epsilon}\right)
                + 2}
            \right)    
            \right\},
    \end{align}
    then the expected optimality gap satisfies
    \begin{align}
        \Exp{F(\x_{k^*})} - F(\opt) \le \epsilon,
    \end{align}
    after
    \begin{align}
        k^*=\bigO{\hat{\kappa}\epsilon^{-1} \log(\epsilon^{-1})}
    \end{align}
    iterations, where $\hat{\kappa} = \hat{L} / \hat{\mu}$, $\sigma^2$ is an upper bound on the trace of the covariance matrix of the gradient estimate $\bs{\upsilon}_k$, and $\alpha \in (0, 1)$ is an error-splitting parameter.
\end{thm}

\begin{proof}
    Given Lemma~\ref{lem:precond}, the update of the preconditioned SGD is
    \begin{align}
        \w_{k+1}^{(k)} = \w_k^{(k)} - \eta_k (\nabla g_k(\w_k^{(k)}) + \sqrt{\Hk^\top} \e_k).
    \end{align}
    Using the upper bound given by $g_k$ being $\hat{L}$-smooth,
    \begin{multline}
        g_k(\w_{k+1}^{(k)}) \le
        g_k(\w_{k}^{(k)}) 
        - \eta_k \inner{\nabla g_k(\w_{k}^{(k)}), \nabla g_k(\w_{k}^{(k)}) + \sqrt{\Hk^\top} \e_k}
        \\
        + \frac{\hat{L} \eta_k^2}{2} \norm{\nabla g_k(\w_{k}^{(k)}) + \sqrt{\Hk^\top} \e_k}^2.
    \end{multline}
    As $\bs{\upsilon}_k$ is an unbiased estimator of $\nabla F(\x_k)$, $\Expc{\e_k}{\w_k^{(k)}} = \bs{0}$. Taking the expectation conditioned on $\w_{k}^{(k)}$,
    \begin{align}
        \Expc{g_k(\w_{k+1}^{(k)})}{\w_{k}^{(k)}}
        &
        - g_k(\w^{*,(k)}) \\
        &\le
        g_k(\w_{k}^{(k)}) 
        - g_k(\w^{*,(k)})
        - \left(\eta_k - \frac{\hat{L} \eta_k^2}{2}\right) \norm{\nabla g_k(\w_{k}^{(k)})}^2
        + \frac{\hat{L} \eta_k^2 \sigma^2}{2 \tilde{\mu}}
        \\
        &\le
        \left(1 - 2\hat{\mu} \left(\eta_k - \frac{\hat{L} \eta_k^2}{2}\right) \right)
        (g_k(\w_{k}^{(k)}) - g_k(\w^{*,(k)}))
        + \frac{\hat{L} \eta_k^2 \sigma^2}{2 \tilde{\mu}}.
    \end{align}
    In the first inequality, we use
    \begin{align}
        \Expc{\norm{\sqrt{\Hk^\top} \e_k}^2}{\w_k^{(k)}} 
        \le \norm{\sqrt{\Hk^\top}}^2 \Expc{\norm{\e_k}^2}{\w_k^{(k)}} 
        \le (\tilde{\mu})^{-1} \sigma^2,
    \end{align}
    and in the last step, we apply the Polyak--{\L}ojasiewicz inequality, which holds due to $g_k$ being $\hat{\mu}$-strongly convex.
    Taking the full expectation,
    \begin{multline}
        \Exp{g_k(\w_{k+1}^{(k)}) - g_k(\w^{*,(k)})}
        \le
        \\
        \left(1 - 
        \underbrace{
            2\hat{\mu} \eta_k \left(1 - \frac{\hat{L} \eta_k}{2}\right)
        }_{r(\eta_k)}
        \right)
        \Exp{g_k(\w_{k}^{(k)}) - g_k(\w^{*,(k)})}
        + \frac{\hat{L} \eta_k^2 \sigma^2}{2 \tilde{\mu}}.
        \label{eq:r_def}
    \end{multline}
    Recalling that $g_i(\w_k^{(i)}) = F(\sqrt{\bs{H}_i} \w_k^{(i)}) = F(\x_k)$, we recover an upper bound on the optimality gap, which we then specialize to a fixed step size of $\eta_k=\eta$, allowing us to unroll the recursion,
    \begin{align}
        \Exp{F(\x_{k+1}) - F(\opt)}
        &\le
        \left(1 - r(\eta_k)\right)
        \Exp{F(\x_k) - F(\opt)}
        + \frac{\hat{L} \eta_k^2 \sigma^2}{2 \tilde{\mu}} \\
        &\le
        \left(1 - r(\eta)\right)^k
        \Exp{F(\x_0) - F(\opt)}
        + \frac{\hat{L} \eta^2 \sigma^2}{2 \tilde{\mu}} \sum_{i=0}^k (1-r(\eta))^i \\
        &\le
        \left(1 - r(\eta)\right)^k
        \Exp{F(\x_0) - F(\opt)}
        + \frac{\hat{L} \eta^2 \sigma^2}{2 \tilde{\mu}} \sum_{i=0}^\infty (1-r(\eta))^i \\
        &=
        \left(1 - r(\eta)\right)^k
        \Exp{F(\x_0) - F(\opt)}
        + \frac{\hat{L} \eta^2 \sigma^2}{2 \tilde{\mu} r(\eta)}.
    \end{align}

    To obtain $\Exp{F(\x_{k+1}) - F(\opt)} \le \epsilon$, we split the error as
    \begin{align}
        \begin{cases}
            \left(1 - r(\eta)\right)^k
            \Exp{F(\x_0) - F(\opt)} &\le \alpha \epsilon \\
            \frac{\hat{L} \eta^2 \sigma^2}{2 \tilde{\mu} r(\eta)} &\le (1 - \alpha) \epsilon,
        \end{cases}
        \label{eq:error_split}
    \end{align}
    for $\alpha \in (0, 1)$.
    For the first inequality to be satisfied, the step size must satisfy $0<\eta<2/\hat{L}$ and the number of iterations required to reach $\alpha \epsilon$ is $k^* = \bigO{(r(\eta))^{-1} \log(\epsilon^{-1})}$.
    The optimal step size is $\eta=1/\hat{L}$, resulting in $r(\eta) = \hat{\mu} / \hat{L}$.
    The second inequality in \eqref{eq:error_split} can be rewritten by substituting $r(\eta)$ as defined in \eqref{eq:r_def},
    \begin{align}
        \eta \le \frac{4}{\hat{L}}
            \underbrace{
            \left(
            \frac{1}{
                \frac{1}{\tilde{\mu} \hat{\mu}} 
                \left(\frac{\sigma^2}{(1-\alpha) \epsilon}\right)
                + 2}
            \right)
            }_{a}.
            \label{eq:a_def}
    \end{align}
    Taking the largest step possible according to \eqref{eq:a_def} and using the definition of $r$ as in \eqref{eq:r_def} yields
    \begin{align}
        r(\eta)
        &=
        (\hat{\kappa})^{-1} 8a(1-2a)
        >
        (\hat{\kappa})^{-1} 8a,
    \end{align}
    where $\hat{\kappa} \coloneqq \hat{L} / \hat{\mu}$.
    Therefore, $k^*=\bigO{\hat{\kappa}a^{-1} \log(\epsilon^{-1})}$. Recalling that $a=\bigO{\epsilon}$ concludes the proof.
\end{proof}

From \eqref{eq:bounds_g}, we observe that as $\Hk$ approaches the inverse of the Hessian of the objective function at iteration $k$, the transformed condition number $\hat{\kappa}$ approaches one. 
If the curvature condition parameter $m$ is set such that $\tilde{\mu} \gg \mu$, $\Hk$ might \emph{not} converge to the true Hessian inverse, but the effect of noise is mitigated. 
Noise is reduced for $\tilde{\mu} > 1$. 
Therefore, a trade-off exists: selecting a large $m$ value avoids noise amplification but results in a higher $\hat{\kappa}$, whereas choosing a smaller $m$ value decreases $\hat{\kappa}$ but increases the risk of noise amplification.
\section{Numeric examples}
First, we will solve a noisy quadratic problem with S-BFGS, BFGS, and SGD.

The primary objective of the first example is to demonstrate the effectiveness of the S-BFGS preconditioner in enhancing the performance of SGD, while mitigating the noise amplification and stability problems encountered when naively employing BFGS equations.
The second example involves training a multinomial logistic regression model on several datasets using L-S-BFGS and two other limited-memory quasi-Newton methods (i.e., SdLBFGS~\cite{wang2017stochastic} and oLBFGS~\cite{mokhtari2015global}).
The aim is to demonstrate that L-S-BFGS retains the advantages of S-BFGS while maintaining favorable scalability for large-dimensional problems. 

Apart from the two parameters that require tuning in S-BFGS, (i.e., the likelihood distribution parameter $\rho$ and the curvature condition parameter $m$), L-S-BFGS requires setting the memory size parameter $r$.
For consistency, we fix the memory size at $r=10$. 
The values of $m$ and $\rho$ are individually optimized for each problem.
Given the known smoothness of the objective function $L$ in the three numerical experiments, an extra curvature condition is imposed in \eqref{eq:extra_curv_condition} with $\tilde{L}=L$. A workstation with 80 processors and an Intel Xeon Gold 6230 processor operating at 2.10 GHz with 252 GB of memory was used to compute the presented numeric results.

\subsection{Quadratic problem}
This simple toy problem was devised to test the Hessian inverse learning capabilities of S-BFGS and BFGS.
Let
\begin{equation}
    f(\bs{x}, \bs{\xi}) = \frac{1}{2} \bs{x}^\top \bs{A} \bs{x} - \bs{x}^\top \bs{1}_d (1 + \bs{x}^\top \bs{\xi}),
\end{equation}
where $\bs{A} \in \bb{R}^{d \times d}$ is a positive-definite matrix with eigenvalues $\lambda$ such that $log_{10}(\lambda) \sim \cl{U}[0, 6]$, $\bs{1}_d$ is a vector of ones with size $d$, and $\bs{\xi} \sim \cl{N}(\bs{0}_d, \bs{\Sigma})$.
The noise covariance matrix is sampled from a Wishart distribution as $\bs{\Sigma} = \cl{W}_p(\bs{V}=10^{-2} \I, n=d)$.
The same realizations of $\bs{A}$ and $\bs{\Sigma}$ are applied when comparing methods.
Moreover, the same starting point $\bs{x}_0 \sim \cl{N}(\bs{0}_d, \I_d)$ is employed.
Here, we use $d=20$ and enforce that the condition number of the problem is $\kappa=10^6$.

The S-BFGS, BFGS, and SGD methods are employed to solve this problem. In all cases, the batch size is $N=10$. 
For S-BFGS and BFGS, the fixed step size is $\eta=0.7$, and they start with a diagonal matrix with all diagonal components equal to $1/L$, where $L=10^6$. 
For S-BFGS, $\rho=100$ and $m = 10^5$. 
For SGD, the fixed step size is $\eta = 1/L$.

Figure~\ref{fig:quadratic} presents the optimality gap vs. iterations for SGD and S-BFGS (top-left), BFGS and S-BFGS (top-right), the eigenvalues profiles for BFGS and S-BFGS (bottom-left), and a measure of distance $\Psi(\cdot) = \tr(\cdot) - \log(\det(\cdot))$ of the preconditioned Hessian to the identity, see Byrd and Nocedal~\cite{byrd1989tool}(bottom-right).
The top-left plot presents the improvement of S-BFGS over SGD, which advances slowly.
In comparison, BFGS using noisy gradients diverges, as can be seen in the top-right plot.
Moreover, BFGS-generated iterates oscillate with a much higher amplitude than those of S-BFGS.
Although the eigenvalues of the BFGS approximation are closer to those of the true inverse Hessian, that might not be an advantage, as revealed in the top-right plot of Figure~\ref{fig:quadratic}.
The more conservative inverse Hessian approximation of S-BFGS seems more suitable for stochastic optimization.
Finally, the bottom-right plot shows that $\Hk$ gets closer to $(\nabla^2 F)^{-1}$ as iterations go.

\begin{figure}[h]
    \includegraphics[width=.95\linewidth]{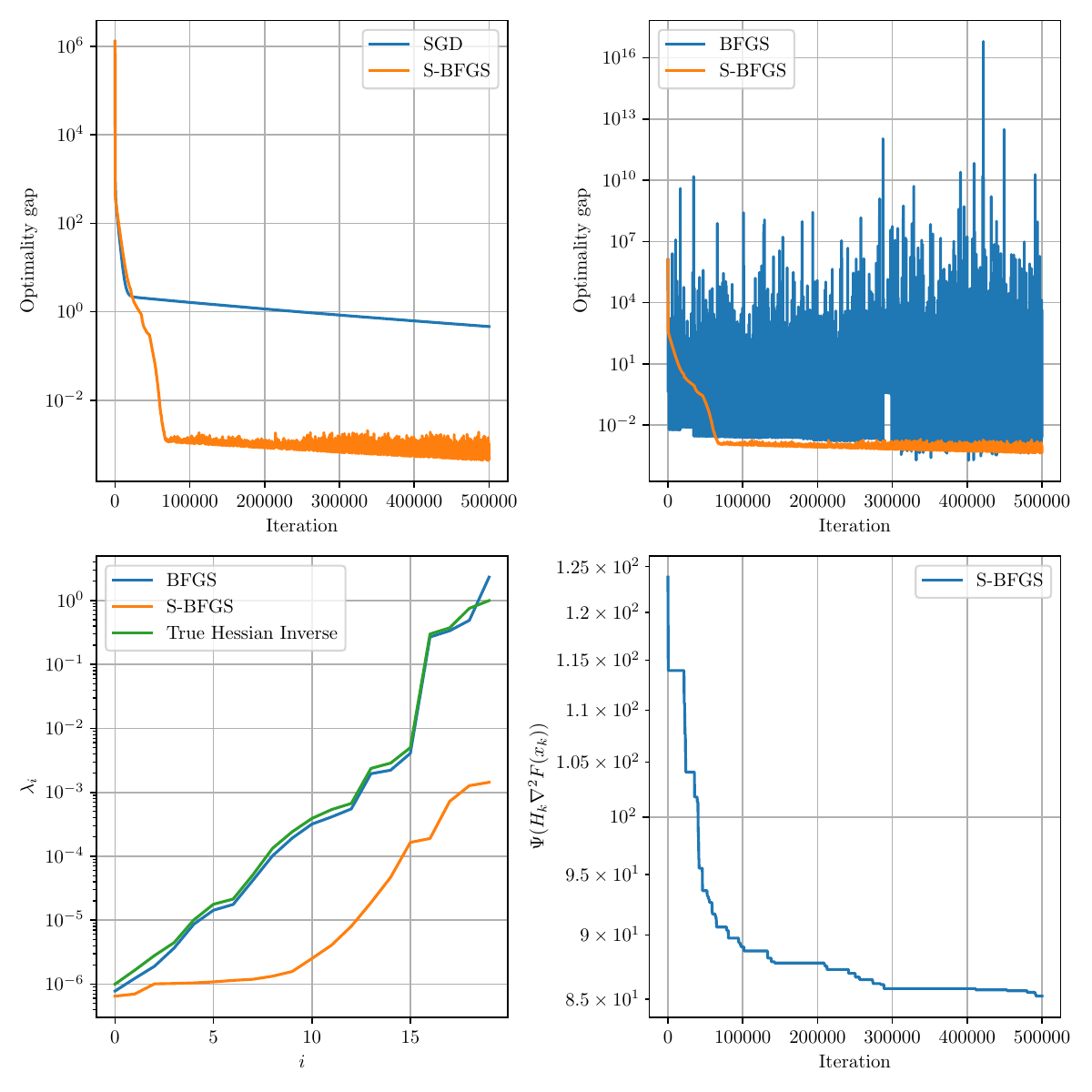}
    \caption{Quadratic problem with a conditioning number of $10^6$: optimality gap vs. iterations for SGD and S-BFGS (top-left), BFGS and S-BFGS (top-right), eigenvalue profiles (bottom-left), and a measure of distance $\Psi$ between $\Hk \nabla^2 F(\x_k)$ and $\I$ (bottom-right).
    The SGD method becomes stuck and does not progress, possibly due to the large condition number of the problem.
    The S-BFGS method converges much faster in the initial iterations until it reaches the asymptotic regime, as predicted by Theorem~\ref{thm:conv}.
    Vanilla BFGS, however, better approximates the Hessian eigenvalues but fails to converge due to noise amplification.
    }
    \label{fig:quadratic}
\end{figure}

\section{Logistic regression}

Here, we compare L-S-BFGS with two other limited-memory quasi-Newton methods, SdLBFGS~\cite{wang2017stochastic} and oLBFGS~\cite{mokhtari2015global}, on training a logistic regression model with $L^2$-regularization.
To make the comparison as fair as possible, we have chosen methods without variance reduction, use fixed step sizes, and a memory size $r=10$ in all cases.
All the datasets were obtained from OpenML\footnote{\url{https://www.openml.org/}}.
In all cases, we employ a regularization parameter $\lambda_{lr} = 10^{-5}$ and a mini-batch of size $N=10$.
For L-S-BFGS, we use a fixed step-size $\eta=0.7$, where for the other methods we tune the fixed step-size for each problem.
A step-size sensibility analysis is presented in the appendices.
Both L-S-BFGS and oLBFGS start with a diagonal matrix with diagonal components set to $1/L$, where $L$ is the smoothness of the objective function, and SdLBFGS starts with the identity matrix, as proposed by \cite{wang2017stochastic}.
The parameters used for each method and each problem are presented in Table~\ref{tab:param}.
\begin{table}[h]
    \centering
    \caption{Logistic regression example: parameters used for different datasets.}
    \label{tab:param}
    \begin{tabular}{llrrrr}
        \toprule
        &                & madelon   & mushrooms & MNIST     & CIFAR-10 \\ \midrule
       \multirow{3}{*}{L-S-BFGS} 
        &$m$   & $10^{-2}$ & $10^{-4}$ & $10^{-4}$ & $10^{-4}$ \\
        &$\rho$      & $10^{-4}$ & $10^2$    & $1.0$       & $10^{-3}$ \\
        &$\eta$      & $0.7$     & $0.7$    & $0.7$       & $0.7$ \\ \midrule
        \multirow{2}{*}{SdLBFGS}
        &$\delta$   & $10^{-3}$ & $10^{-2}$          & $10^{-2}$ & $10^{-2}$ \\
        &$\eta$      & $0.1$    & $5 \times 10^{-2}$ & $10^{-2}$       & $10^{-1}$ \\ \midrule
        \multirow{1}{*}{oLBFGS}
        &$\eta$   & $5 \times 10^{-5}$ & $10^{-3}$ & $3 \times 10^{-3}$ & $10^{-2}$ \\
        \bottomrule
    \end{tabular}
\end{table}

All methods were run 50 times each with the same starting point $\bs{x}_0=\bs{O}_d$ to provide statistical confidence intervals.
Figure~\ref{fig:logreg} presents the optimality gap vs. both epochs (right) and runtimes (left) for the logistic regression examples.
Moreover, Figure~\ref{fig:logreg} presents the confidence intervals of $50\%$ and $90\%$. 
The L-S-BFGS method can converge faster than the baseline methods in all tested cases, even with condition numbers on the order of $10^{12}$. 
Notably, L-S-BFGS works well even with a short memory of $r=10$, a small fraction of the dimensionality of the problems, which can be as large as $d=30,720$ in CIFAR-10.

\begin{figure*}[h!]
    \includegraphics[width=.49\linewidth]{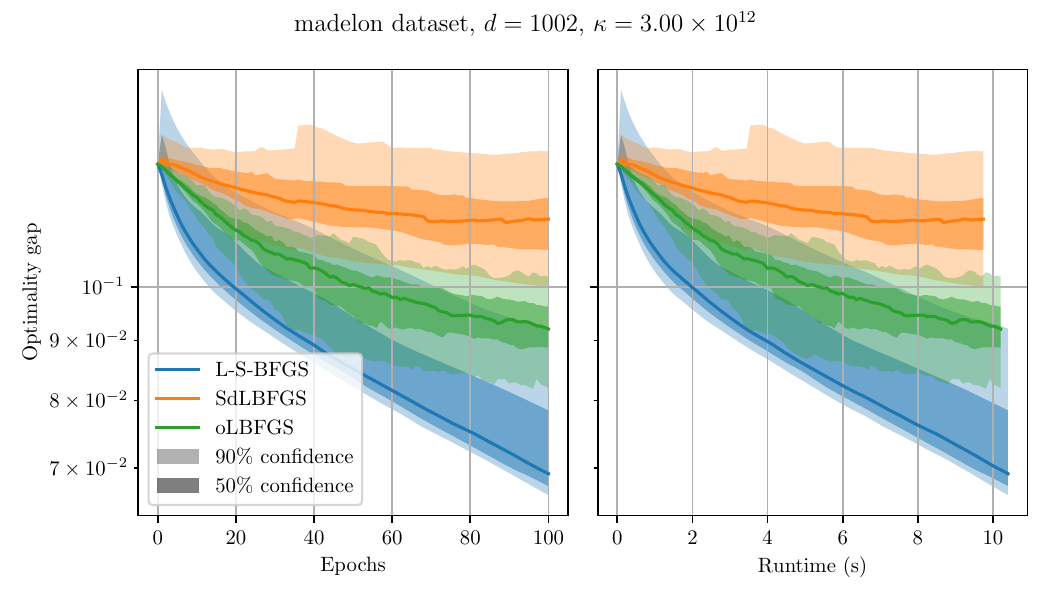}
    \includegraphics[width=.49\linewidth]{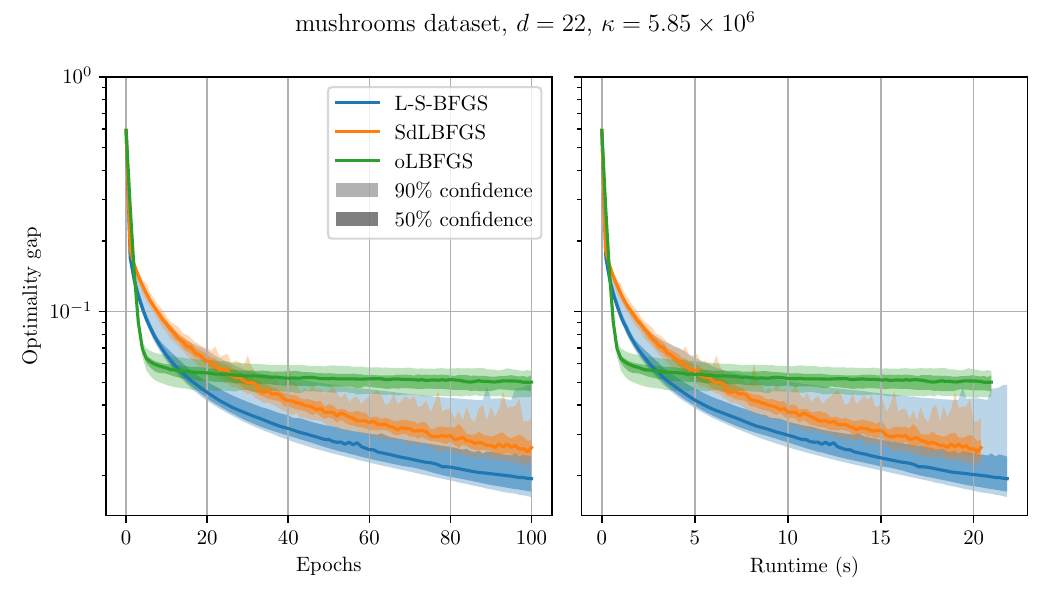}
    \includegraphics[width=.49\linewidth]{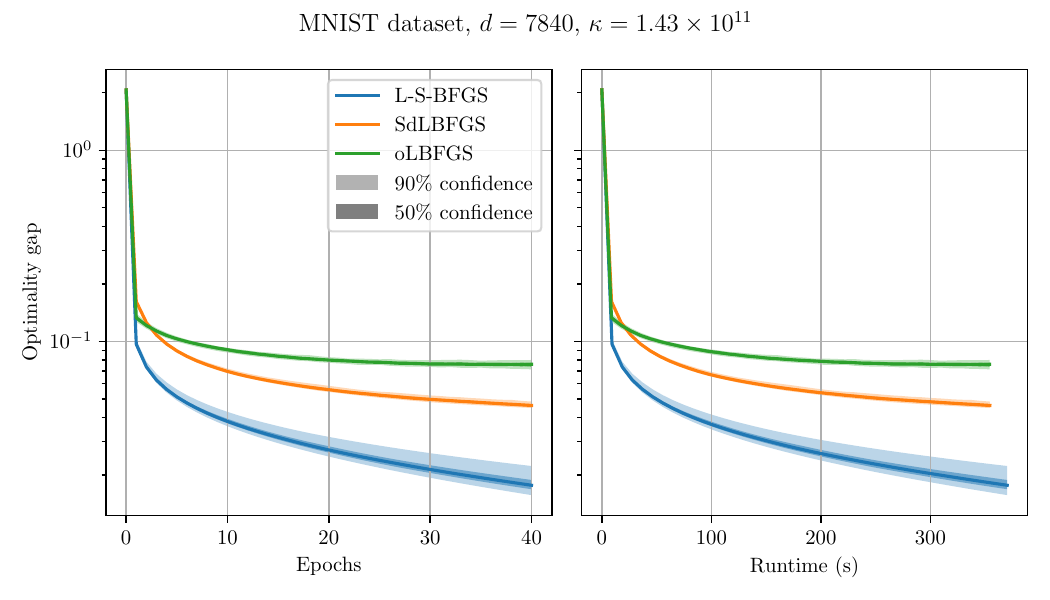}
    \includegraphics[width=.49\linewidth]{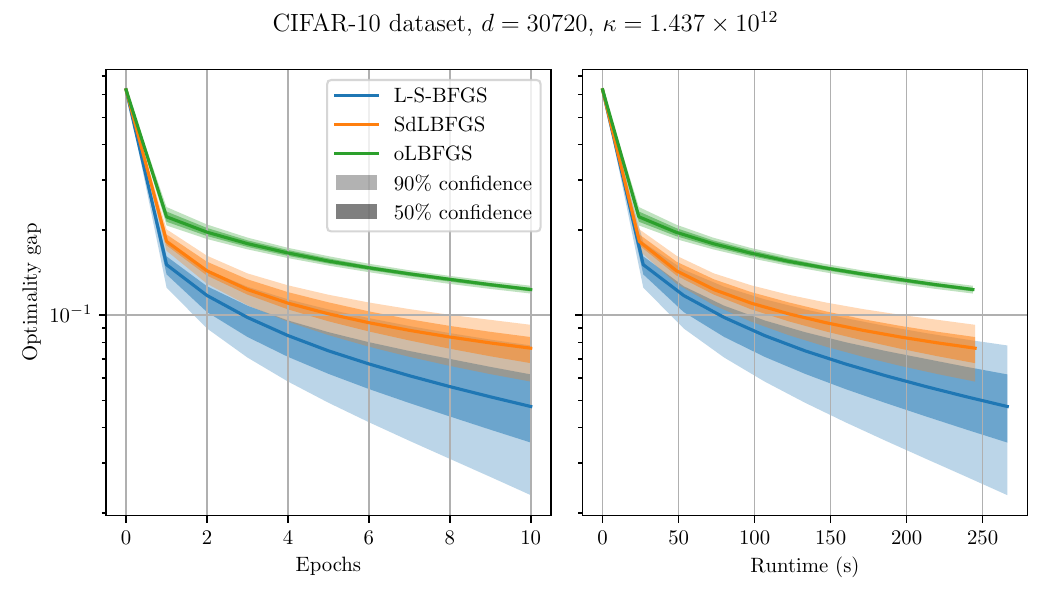}
    \caption{Logistic regression: Optimality gap vs. epoch for 50 independent runs of L-S-BFGS, SdLBFGS, and oLBFGS. Light shade indicates the confidence interval of $90\%$; darker shade presents the confidence interval of $50\%$; and solid lines indicate median values.
    The robustness of L-S-BFGS to noise allows larger step sizes, offering an advantage in comparison to the baseline methods. 
    As SdLBFGS also has mechanisms to control noise, it performs better than oLBFGS, which requires step sizes as small as $3 \times 10^{-3}$ to converge in the MNIST experiment.
    }
    \label{fig:logreg}
\end{figure*}

\section*{Acknowledgments}
This work was supported by the KAUST Office of Sponsored Research (OSR) under Award No. URF/1/2584-01-01 and the Alexander von Humboldt Foundation.

\section{Conclusion}
In this work, we propose a new methodology for deriving quasi-Newton methods for stochastic optimization.
We use Bayesian inference to incorporate noisy gradient observations into our approximation of the inverse Hessian by building a prior distribution for the inverse Hessian and a likelihood of observing the curvature pairs.
Furthermore, we derived a stochastic counterpart to BFGS that converges to the deterministic BFGS as the gradient noise diminishes by selecting an appropriate prior distribution.

Although the computational cost of S-BFGS is comparable to that of BFGS, our L-S-BFGS algorithm has a cost of $\mathcal{O}(dm^2)$ for a memory size $m$, compared to $\mathcal{O}(dm)$ for L-BFGS.
Nevertheless, in practice, L-S-BFGS outperforms other baseline quasi-Newton methods on large-scale problems, as it is more robust to noise, which allows it to take larger steps.

The present work makes several contributions.
On a theoretical level, this work introduces a new class of quasi-Newton methods for stochastic optimization using Bayesian inference and provides a convergence analysis of the preconditioned stochastic gradient descent (SGD). 
On a practical level, this work evaluates these methods and demonstrates their performance improvements in numerical experiments.

Despite the robustness of our proposed methods, setting parameters is still a burden.
The two parameters required by our methods are not dimensionless, meaning their choice depends on the specific problem at hand. 
Future research should explore coupling our preconditioner with existing SGD variants, such as SVRG~\cite{johnson2013accelerating}, SARAH~\cite{nguyen2017sarah}, or MICE~\cite{carlon2024multi}.
In these variants, gradient differences are already computed and could be used to enhance our quasi-Newton methods.

\appendix
\section{Step size sensitivity analysis}

The choice of step-size is critical in the convergence of the proposed method and other quasi-Newton methods used as baselines. 
Figure~\ref{fig:steps} presents the comparison of a single run using various step sizes for the logistic regression problem.

\begin{figure*}[h]
    \includegraphics[width=.48\linewidth]{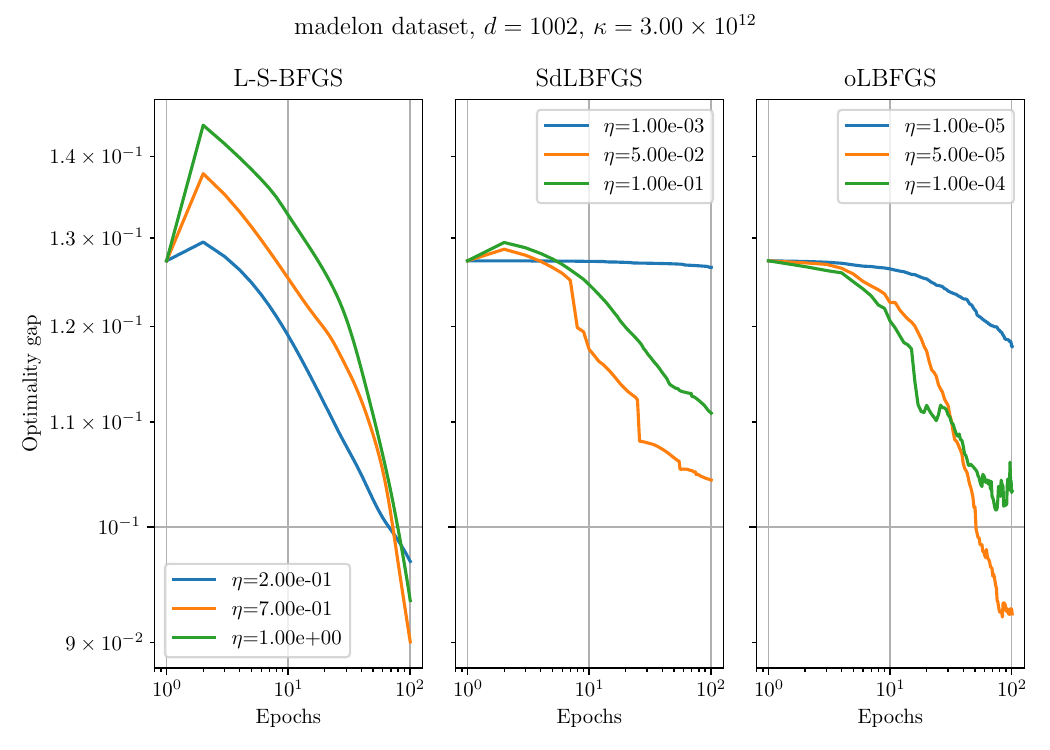}
    \includegraphics[width=.48\linewidth]{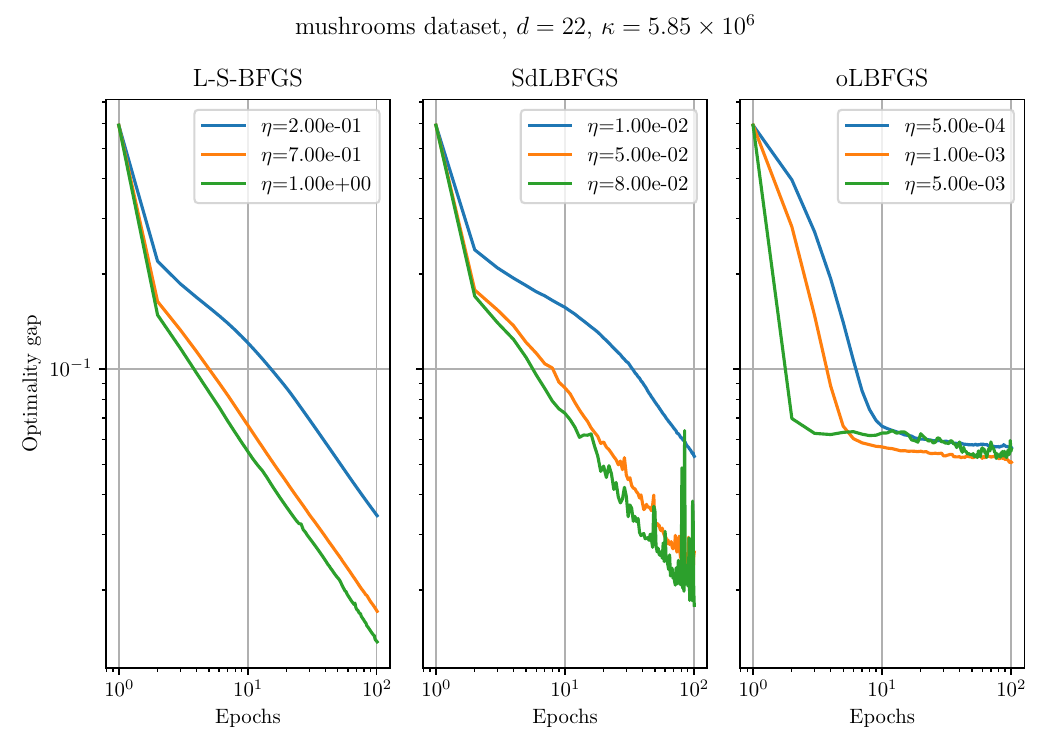} \\
    \includegraphics[width=.48\linewidth]{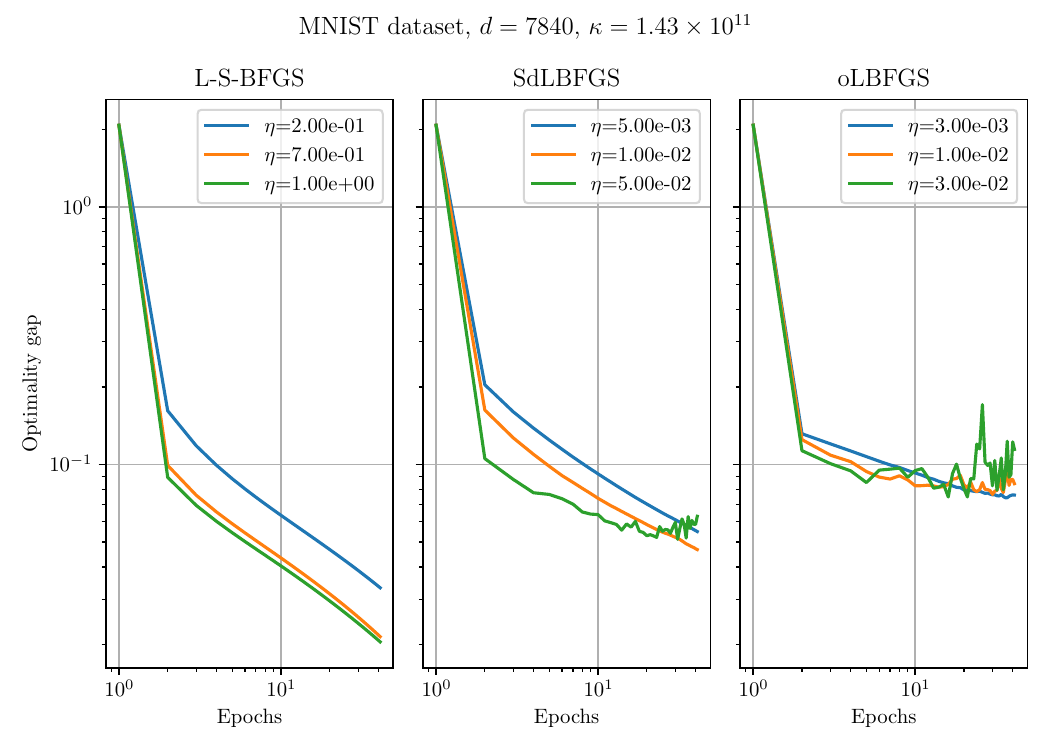}
    \includegraphics[width=.48\linewidth]{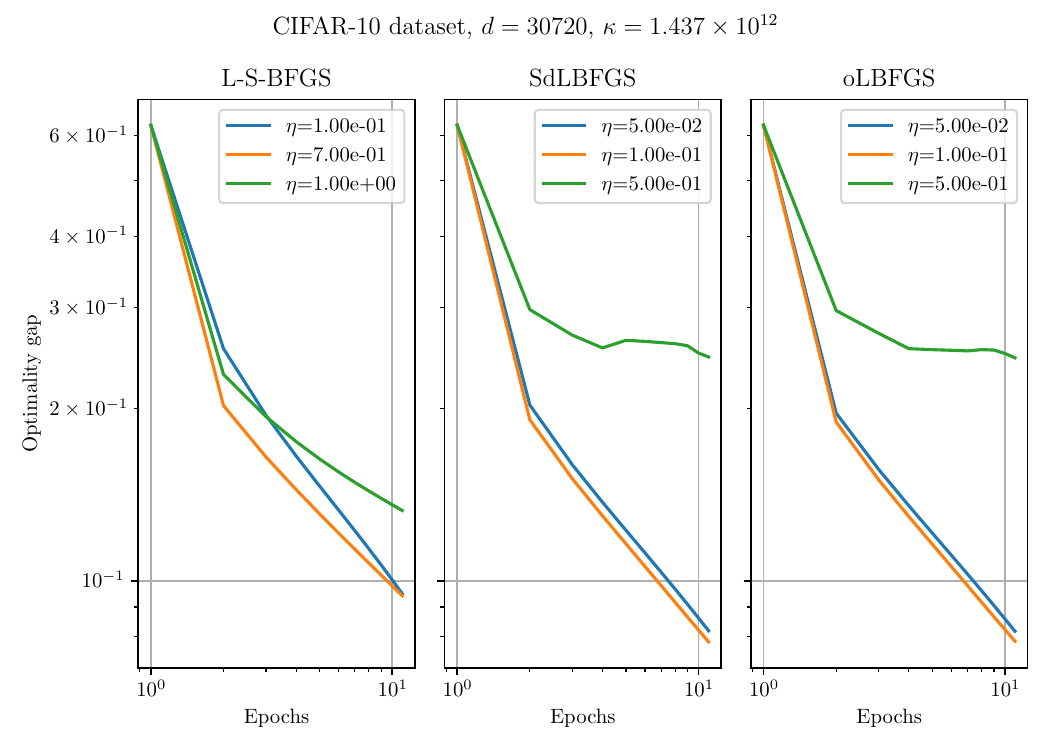}
    \caption{Optimality gap versus epochs for the different methods using different step sizes.
    }
    \label{fig:steps}
\end{figure*}

The sensitivity analysis confirms that L-S-BFGS is more robust to large choices of step sizes than the baseline stochastic quasi-Newton methods.

\bibliographystyle{siamplain}
\bibliography{references}
\end{document}